\documentclass[10pt,a4paper]{article}
\usepackage[latin1]{inputenc}
\usepackage{amsmath}
\usepackage{amsfonts}
\usepackage{amssymb}
\usepackage{graphicx}
\usepackage{amsthm} 	
\usepackage{hyperref}
\usepackage{todonotes}
\usetikzlibrary{backgrounds}	
\usepackage[toc]{appendix}

\usepackage{graphicx,xcolor}	

\usepackage{soul}

\usepackage{xcolor}
\pagecolor{white}

\usepackage{esint}				

\usepackage[a4paper,top=2.1cm,bottom=2.60cm,left=3.8cm,right=3.8cm]{geometry} 

\allowdisplaybreaks

\newtheorem{theorem}{Theorem}[section]
\newtheorem{proposition}[theorem]{Proposition}

\newtheorem{remark}[theorem]{Remark}
\newtheorem{lemma}[theorem]{Lemma}

\title{Integral equation method for a Robin-type traction problem in a periodic domain}
\author{ Matteo Dalla Riva  \thanks{Dipartimento di Ingegneria, Universit\`a degli Studi di Palermo, Viale delle Scienze, Ed. 8, 90128 Palermo, Italy. E-mail: matteo.dallariva@unipa.it } ,  Gennady Mishuris \thanks{Department of Mathematics, Aberystwyth University, Ceredigion, Aberystwyth, SY23 3BZ Wales, UK. E-mail: ggm@aber.ac.uk} , Paolo Musolino\thanks{Dipartimento di Scienze Molecolari e Nanosistemi, Universit\`a Ca' Foscari Venezia, via Torino 155, 30172 Venezia Mestre, Italy. E-mail: paolo.musolino@unive.it}\ \thanks{{\it Corresponding author.}}}

\date{June 23, 2022}

\begin{document}

\maketitle

\noindent

{\bf Abstract:}    In this note, we consider a Robin-type traction  problem for a linearly elastic body occupying an infinite periodically perforated domain. After proving the uniqueness of the solution we use periodic elastic layer potentials to show that the solution can be written as the sum of a single layer potential, a constant function and a linear function of the space variable. The density of the periodic single layer potential and the constant are identified as the unique solutions of a certain integral equation.

\vspace{9pt}

\vspace{9pt}


\section{Introduction}\label{introd}

The analysis of problems for perforated plates and other porous materials started at the beginning of the 20th century  (see, e.g., \cite{Na35, Ho35, Ho52, BaHi60, GoJa65}) and, at least in the first stage of development, was mainly  dedicated to problems with classic boundary conditions of Dirichlet and Neumann types (see Mityushev {\it et al.} \cite{MiAnGl22} for a thorough review). Recent advances in material sciences, however, brought the attention to problems with different kinds of boundary conditions, and even to problems with boundary conditions of nonlinear type. 
 
One example comes from the employment of porous coating,  or interfacial coatings in the case of inclusions, that are used to produce metamaterials  with properties  that are not typically present in nature and to enhance desirable   characteristics, such as the corrosion resistance, biocompatibility, biodegradation, and so on. When the thickness of the coating is negligibly smaller than the characteristic size of the pores, while its material properties (elastic, thermal, magneto-electric, etc.) are, in a sense,  weaker/softer  in comparison with those of the main composite material (matrix), the corresponding mathematical problems may degenerate into periodic boundary value problems with conditions of Robin type (see \cite{KlMo98, MiBe99, AnAvKoMo01, Mi01, Mi04, MiMoMo06, Be06, SoPiMiMi15, XuTiXi20}). To the best of our knowledge, these problems seem to be rarely considered in literature, and, therefore, we show in this note how we can effectually employ an integral equations approach.

Indeed, integral equation methods have proven to be a very useful tool to deal with problems that are relevant in the applications. The literature is massive and a complete list of applications may range from scattering theory and inverse problems (as, for example, in Ammari and Kang \cite{AmKa07}, Castro {\it et al.~}\cite{CaDuKa11}, Colton and Kress \cite{CoKr83}, Kirsch and Hettlich \cite{KiHe15}), to elasticity and thermoelasticity (as in Duduchava \cite{Du82}, Duduchava  {\it et al.~}\cite{DuNaSh95i, DuNaSh95ii}, Kupradze {\it et al.~}\cite{KuGeBaBu79}), fluid mechanics (for example in Kohr {\it et al.~}\cite{KoLaWe14}), and to the composite materials that are the subject of the present note (see also Chkadua {\it et al.~}\cite{ChMiNa11},  Duduchava {\it et al.~}\cite{DuSaWe99}). 

In addition, the study of composite materials  often boils down  to the analysis of boundary value problems  in periodic domains  (see, e.g., Milton \cite[Ch.~1]{Mi02}, Movchan {\it et al.}~\cite{MoMoPo02}). In dimension 2, these problems can be tackled with complex variable techniques  (see, e.g., Kapanadze {\it et al.~}\cite{KaMiPe15a, KaMiPe15b},   
  Dryga\'s, {\it et al.~}\cite{DrGlMiNa20}, 
  Gluzman {\it et al.~}\cite{GlMiNa18},  Kapanadze {\it et al.~}\cite{KaMiPe16}, Mityushev {\it et al.~}\cite{MiObPeRo08}). In higher dimensions, complex variable techniques are in general not an option, but integral equation methods provide an effective alternative.

We now describe the problem of this note: We consider a linearly elastic body that occupies an infinite periodically perforated domain. On the boundary of the body we set a Robin-type traction condition,  which prescribes a linear relation between the traction applied to the boundary  and the displacement of the boundary points.  After introducing the corresponding system of differential equations and boundary conditions, we will analyze it by means of the integral equation method.  
We start by introducing the geometric setting. We fix once for all
\[
n\in {\mathbb{N}}\setminus\{0,1 \}\,,\qquad  (q_{11},\dots,q_{nn})\in]0,+\infty[^{n}\,.
\]
Here ${\mathbb{N}}$ denotes the 
set of natural numbers including $0$. We denote by $Q$ the fundamental periodicity cell defined by
\begin{equation}\label{Q}
Q\equiv\Pi_{j=1}^{n}]0,q_{jj}[
\end{equation}
 and by $\nu_Q$ the outward unit normal to $\partial Q$, where it exists. We denote by $q$ the diagonal matrix defined by
\begin{equation}\label{q}
q\equiv \left(
\begin{array}{cccc}
q_{11} &   0 & \dots & 0   
\\
0          &q_{22} &\dots & 0
\\
\dots & \dots & \dots & \dots  
\\
0& 0 & \dots & q_{nn}
\end{array}\right)\, .
\end{equation}
We will define our periodically perforated domain, by removing from $\mathbb{R}^n$ congruent copies of a bounded domain of class $C^{m,\alpha}$. Therefore, we fix once and for all
\[
m\in {\mathbb{N}}\setminus\{0\}\,,\qquad\alpha\in]0,1[\,.
\]
Then we assume that
\[ 
\text{${\Omega_Q}$ is a bounded open subset of ${\mathbb{R}}^{n}$ of class $C^{m,\alpha}$ such that $\overline{\Omega_Q}\subseteq Q$,}
\]
and we define the periodic domains
\[
{\mathbb{S}} [\Omega_Q]\equiv 
\bigcup_{z\in{\mathbb{Z}}^{n} }(qz+\Omega_Q)=q{\mathbb{Z}}^{n}+\Omega_Q\,,\qquad
{\mathbb{S}} [\Omega_Q]^{-}\equiv {\mathbb{R}}^{n}\setminus\overline{\mathbb{S}[\Omega_Q]}\,.\nonumber
\] 

We now introduce a Robin boundary value problem in ${\mathbb{S}} [\Omega_Q]^{-}$ for the Lam\'e equations in ${\mathbb{S}} [\Omega_Q]$. To do so, we denote by $T$ the function from $ ]1-(2/n),+\infty[\times M_n(\mathbb{R})$ to $M_n(\mathbb{R})$ defined by 
\[
T(\omega,A)\equiv (\omega-1)(\mathrm{tr}A)I_n+(A+A^t) \qquad \forall  \omega \in ]1-(2/n),+\infty[\,,\ A \in M_n(\mathbb{R})\,.
\] 
Here $M_n(\mathbb{R})$ denotes the space of $n\times n$ matrices with real entries, $I_n$ denotes the $n\times n$ identity matrix, $\mathrm{tr}A$ and $A^t$ denote the trace and the transpose matrix of $A$, respectively. We note that $(\omega-1)$ plays the role of the ratio between the first and second Lam\'e constants and that the classical linearization of the Piola Kirchoff tensor equals the second Lam\'e constant times $T(\omega,\cdot)$ (cf., {e.g.}, Kupradze {\it et al.~}\cite{KuGeBaBu79}). 
Then we consider also the following assumptions:
\[
\begin{split}
&\text{Let $B \in M_n(\mathbb{R})$.}\\
&\text{Let $a,b \in C^{m-1,\alpha}(\partial \Omega_Q,M_n(\mathbb{R}))$ be such that:}\\
&\text{\ \  $\bullet$ $\mathrm{det}\,  a(x) \neq 0$ for all $x \in \partial \Omega_Q$,}\\
&\text{\ \  $\bullet$ $\xi^t a^{-1}(x)b(x)\xi \leq 0$ for all $x \in \partial \Omega_Q$, $\xi \in \mathbb{R}^n$, and $\mathrm{det}\int_{\partial \Omega_Q} a^{-1}b\,d\sigma\neq 0$,}\\
&\text{\ \  $\bullet$ there exists $x_0 \in \partial \Omega_Q$ such that $\mathrm{det}\,b(x_0)\neq 0$.}\\
&\text{Let $g \in C^{m-1,\alpha}(\partial \Omega_Q, \mathbb{R}^n)$.}
\end{split}
\]
As we shall see the conditions on $a$ and $b$ play a crucial role for the existence and uniqueness of a solution.
Then we take
\[
\omega \in ]1-(2/n),+\infty[\, ,
\]
and we consider the following Robin boundary value problem 
 \begin{equation}\label{bvp}
 \left \lbrace 
 \begin{array}{ll}
  \mathrm{div}\, T(\omega, Du )= 0 & \mathrm{in}\  {\mathbb{S}} [\Omega_Q]^-\,, \\
u(x+qe_j) =u(x) +Be_j&  \textrm{$\forall x \in \overline{\mathbb{S}[\Omega_Q]^{-}}, \forall j\in \{1,\dots,n\}$}, \\
a(x)T(\omega,Du(x))\nu_{\Omega_Q}(x)+ b(x)u(x)=g(x) & \textrm{$\forall x \in \partial \Omega_Q$}\,,
 \end{array}
 \right.
 \end{equation}
where $\{e_1,\dots,e_n\}$ denotes the canonical basis of $\mathbb{R}^n$ and $\nu_{\Omega_Q}$ denotes the outward unit normal to $\partial \Omega_Q$. 

The aim of this note is to prove that the solution to problem \eqref{bvp} exists and is unique, then to convert problem \eqref{bvp} into an equivalent integral equation,  and finally to show that the solution can be written as the application of a specific integral operator to a density function that is the solution of a certain boundary integral equation. The long term goal is to provide some tools that can be used to analyze perturbation problems for the Lam\'e equations in periodic domains by means of  the so-called {\it Functional Analytic Approach} (see~\cite{DaLaMu21}).  Indeed, the Functional Analytic Approach has been largely used to study periodic problems for the Laplace equations, also in connection to the analysis of effective properties (see, e.g., \cite{DaLuMuPu21, DaMuPu19, LuMu20, LuMuPu19}),  but the application to perturbation problems for the Lam\'e equations in periodic domains is more limited (see the papers \cite{DaMu14} and \cite{FaLuMu21}). In particular, although several techniques are available for the analysis of this type of problems, in this note we develop the tools we wish to use to carry out the analysis done in \cite{MuMi18} for the analysis of degenerating boundary conditions in the case of a (non-periodic) mixed problem for the Laplace equation. More precisely, with the results of the present note we wish to study the asymptotic behavior of the solution of a problem like \eqref{bvp} where the coefficient in front of $u$ tends to $0$.

\section{Some notation}\label{not}

We  denote the norm on 
a   normed space ${\mathcal X}$ by $\|\cdot\|_{{\mathcal X}}$. Let 
${\mathcal X}$ and ${\mathcal Y}$ be normed spaces. We equip the  
space ${\mathcal X}\times {\mathcal Y}$ with the norm defined by 
$\|(x,y)\|_{{\mathcal X}\times {\mathcal Y}}\equiv \|x\|_{{\mathcal X}}+
\|y\|_{{\mathcal Y}}$ for all $(x,y)\in  {\mathcal X}\times {\mathcal 
Y}$, while we use the Euclidean norm for ${\mathbb{R}}^{n}$.
 We denote by $\mathcal{L}(\mathcal{X},\mathcal{Y})$ the space of linear and continuous maps from $\mathcal{X}$ to $\mathcal{Y}$, equipped with its usual norm of the uniform convergence on the unit sphere of $\mathcal{X}$. We denote by $I$ the identity operator. The inverse function of an 
invertible function $f$ is denoted $f^{(-1)}$, as opposed to the 
reciprocal of a real-valued function $g$, or the inverse of a 
matrix $B$, which are denoted $g^{-1}$ and $B^{-1}$, respectively.     If $B$ is a 
matrix, then 
 $B_{ij}$ denotes 
the $(i,j)$ entry of $B$. If $x\in\mathbb{R}^n$, then $x_{j}$ denotes the $j$-th coordinate of $x$ and 
$|x|$ denotes the Euclidean modulus of $ x$. A  dot ``$\cdot$'' denotes the inner product in ${\mathbb R}^{n}$. For all $R>0$ and all $x\in{\mathbb{R}}^{n}$ we denote by ${\mathbb{B}}_{n}( x,R)$ the ball $\{
y\in{\mathbb{R}}^{n}:\, | x- y|<R\}$.   If 
$\mathcal{S}$ is a subset of ${\mathbb {R}}^{n}$, then $\overline{\mathcal{S}}$ 
denotes the 
closure of $\mathcal{S}$ and $\partial\mathcal{S}$ denotes the boundary of $\mathcal{S}$. If we further assume that $\mathcal{S}$ is measurable then $|\mathcal{S}|$ denotes the $n$-dimensional measure of $\mathcal{S}$. Let $q$ be as in definition \eqref{q}. Let $\mathcal{P}$ be a subset of $\mathbb{R}^n$ such that $x+qz \in \mathcal{P}$ for all $x \in \mathcal{P}$ and for all $z \in {\mathbb{Z}^n}$. We say that a function $f$ on $\mathcal{P}$ is $q$-periodic if
\[
f(x+qz)=f(x) \qquad \forall x \in \mathcal{P}\,,\quad \forall z \in \mathbb{Z}^n\, .
\]
Let $\mathcal{O}$ be an open 
subset of ${\mathbb{R}}^{n}$. Let $k\in\mathbb{N}$.  The space of $k$ times continuously 
differentiable real-valued functions on $\mathcal{O}$ is denoted by 
$C^{k}(\mathcal{O},{\mathbb{R}})$, or more simply by $C^{k}(\mathcal{O})$. If $f\in C^{k}(\mathcal{O})$  then $\nabla f$ denotes the gradient $\left(\frac{\partial f}{\partial
x_1},\dots,\frac{\partial f}{\partial
x_n}\right)$ which we think as a column vector. Let $r\in {\mathbb{N}}\setminus\{0\}$. Let $f\equiv(f_1,\dots,f_r)\in \left(C^{k}(\mathcal{O})\right)^{r}$. Then $Df$ denotes the Jacobian matrix
$\left(\frac{\partial f_s}{\partial
x_l}\right)_{  (s,l)\in\{1,\dots,r\}\times\{1,\dots,n\}}$.  Let  $\eta\equiv
(\eta_{1},\dots ,\eta_{n})\in{\mathbb{N}}^{n}$, $|\eta |\equiv
\eta_{1}+\dots +\eta_{n}  $. Then $D^{\eta} f$ denotes
$\frac{\partial^{|\eta|}f}{\partial
x_{1}^{\eta_{1}}\dots\partial x_{n}^{\eta_{n}}}$.  The
subspace of $C^{k}(\mathcal{O})$ of those functions $f$ whose derivatives $D^{\eta }f$ of
order $|\eta |\leq k$ can be extended with continuity to 
$\overline{\mathcal{O}}$  is  denoted $C^{k}(
\overline{\mathcal{O}})$. Let $\beta\in]0,1[$. The
subspace of $C^{k}(\overline{\mathcal{O}}) $  whose
functions have $k$-th order derivatives that are uniformly
H\"{o}lder continuous in $\overline{\mathcal{O}}$ with exponent  $\beta$ is denoted $C^{k,\beta} (\overline{\mathcal{O}})$  
(cf., {e.g.},~Gilbarg and 
Trudinger~\cite{GiTr83}). If $f\in C^{0,\beta}(\overline{\mathcal{O}})$, then its  $\beta$-H\"{o}lder constant $|f:\overline{\mathcal{O}}|_{\beta}$ is defined as  
 $\sup\left\{
\frac{|f( x )-f( y)|}{| x- y|^{\beta}
}: x, y\in \overline{\mathcal{O}} ,  x\neq
 y\right\}$. The subspace of $C^{k}(\overline{\mathcal{O}}) $ of those functions $f$ such that $f_{|\overline{(\mathcal{O}\cap{\mathbb{B}}_{n}(0,R))}}\in
C^{k,\beta}(\overline{(\mathcal{O}\cap{\mathbb{B}}_{n}(0,R))})$ for all $R\in]0,+\infty[$ is denoted $C^{k,\beta}_{{\mathrm{loc}}}(\overline{\mathcal{O}}) $.  Let 
$\mathcal{S}\subseteq {\mathbb{R}}^{r}$. Then $C^{k
,\beta }(\overline{\mathcal{O}} ,\mathcal{S})$ denotes
$\left\{f\in \left( C^{k,\beta} (\overline{\mathcal{O}})\right)^{r}:\ f(
\overline{\mathcal{O}})\subseteq \mathcal{S}\right\}$. Then we set
\[
\begin{split}
C^{k}_{b}(\overline{\mathcal{O}},\mathbb{R}^n)\equiv
\{
u\in C^{k}(\overline{\mathcal{O}},\mathbb{R}^n):\,
D^{\eta}u\ {\mathrm{is\ bounded}}\ \textrm{for all }\eta\in {\mathbb{N}}^{n}\
{\mathrm{with}}\ |\eta|\leq k
\}\,,
\end{split}
\]
and we equip $C^{k}_{b}(\overline{\mathcal{O}},\mathbb{R}^n)$ with its usual  norm
\[
\|u\|_{ C^{k}_{b}(\overline{\mathcal{O}},\mathbb{R}^n) }\equiv
\sum_{\eta \in \mathbb{N}^n\, , \ |\eta|\leq k}\sup_{x\in \overline{\Omega} }|D^{\eta}u(x)|\,.
\]
We define
\[
\begin{split}
C^{k,\beta}_{b}(\overline{\mathcal{O}},\mathbb{R}^n)\equiv
\{
u\in C^{k,\beta}(\overline{\mathcal{O}},\mathbb{R}^n):\,
D^{\eta}u\ {\mathrm{is\ bounded}}\ \textrm{for all }\eta\in {\mathbb{N}}^{n}\
{\mathrm{with
}}\ |\eta|\leq k
\}\,,
\end{split}
\]
and we equip $C^{k,\beta}_{b}(\overline{\mathcal{O}},\mathbb{R}^n)$ with its usual  norm
\[
\|u\|_{ C^{k,\beta}_{b}(\overline{\mathcal{O}},\mathbb{R}^n) }\equiv
\sum_{\eta \in \mathbb{N}^n\, ,\ |\eta|\leq k}\sup_{x\in \overline{\mathcal{O}} }|D^{\eta}u(x)|
+\sum_{\eta \in \mathbb{N}^n\, ,\ |\eta| = k}|D^{\eta}u: \overline{\mathcal{O}} |_{\beta}\, .
\]

Let $\mathcal{O} $ be a bounded
open subset of  ${\mathbb{R}}^{n}$. Then $C^{k}(\overline{\mathcal{O}} )$ 
and $C^{k,\beta}(\overline{\mathcal{O}})$ equipped with their usual norm are well known to be 
Banach spaces  (cf., {e.g.}, Troianiello~\cite[\S 1.2.1]{Tr87}). 
We say that a bounded open subset $\mathcal{O}$ of ${\mathbb{R}}^{n}$ is of class 
$C^{k}$ or of class $C^{k,\beta}$, if its closure is a 
manifold with boundary imbedded in 
${\mathbb{R}}^{n}$ of class $C^{k}$ or $C^{k,\beta}$, respectively
 (cf., {e.g.}, Gilbarg and Trudinger~\cite[\S 6.2]{GiTr83}). 
  For standard properties of functions 
in Schauder spaces, we refer the reader to Gilbarg and 
Trudinger~\cite{GiTr83} and to Troianiello~\cite{Tr87}
(see also Lanza de Cristoforis  \cite[\S 2, Lem.~3.1, 4.26, Thm.~4.28]{La91}, 
 Lanza de Cristoforis   and Rossi \cite[\S 2]{LaRo04}).
If $\mathcal{M}$ is a manifold  imbedded in 
${\mathbb{R}}^{n}$ of class $C^{k,\beta}$ with $k\ge 1$, then we can define the Schauder spaces also on $\mathcal{M}$ by 
exploiting the local parametrization. In particular, if $\mathcal{O}$ is a bounded open set of class $C^{k,\beta}$ with $k\ge 1$, then we can consider 
the space $C^{l,\beta}(\partial\mathcal{O})$ on $\partial\mathcal{O}$
with $l \in \{0,\dots,k\}$ and the trace operator from $C^{l,\beta}(\overline{\mathcal{O}})$ to
$C^{l,\beta}(\partial\mathcal{O})$ is linear and continuous.
 Now let $Q$ be as in definition \eqref{Q}. If ${\mathcal{S}_Q}$ is an arbitrary subset of ${\mathbb{R}}^{n}$  such that
 $\overline{\mathcal{S}_Q}\subseteq Q$, then we define
\[
{\mathbb{S}} [{\mathcal{S}_Q}]\equiv 
\bigcup_{z\in{\mathbb{Z}}^{n} }(qz+{\mathcal{S}_Q})=q{\mathbb{Z}}^{n}+{\mathcal{S}_Q}\,,\qquad
{\mathbb{S}} [{\mathcal{S}_Q}]^{-}\equiv {\mathbb{R}}^{n}\setminus\overline{\mathbb{S}[{\mathcal{S}_Q}]}\,.\nonumber
\] 
We note that if $\mathbb{R}^n\setminus \overline{\mathcal{S}_Q}$ is connected, then $\mathbb{S}[{\mathcal{S}_Q}]^{-}$ is also connected.  If ${\Omega_Q}$ is an open subset of $\mathbb{R}^n$ such that $\overline{\Omega_Q} \subseteq Q$, then we denote by $C^{k}_{q}(\overline{\mathbb{S}[\Omega_Q]},\mathbb{R}^n )$, $C^{k,\beta}_{q}(\overline{\mathbb{S}[\Omega_Q]},\mathbb{R}^n )$, $C^{k}_{q}(\overline{\mathbb{S}[\Omega_Q]^{-}},\mathbb{R}^n )$, and $C^{k,\beta}_{q}(\overline{\mathbb{S}[\Omega_Q]^{-}},\mathbb{R}^n )$ the subsets of the $q$-periodic functions belonging to $C^{k}_{b}(\overline{\mathbb{S}[\Omega_Q]},\mathbb{R}^n )$,  to $C^{k,\beta}_{b}(\overline{\mathbb{S}[\Omega_Q]},\mathbb{R}^n )$,  to $C^{k}_{b}(\overline{\mathbb{S}[\Omega_Q]^{-}},\mathbb{R}^n )$, and to $C^{k,\beta}_{b}(\overline{\mathbb{S}[\Omega_Q]^{-}},\mathbb{R}^n )$, respectively. \\ We regard  $C^{k}_{q}(\overline{\mathbb{S}[\Omega_Q]},\mathbb{R}^n )$, $C^{k,\beta}_{q}(\overline{\mathbb{S}[\Omega_Q]},\mathbb{R}^n )$, $C^{k}_{q}(\overline{\mathbb{S}[\Omega_Q]^{-}},\mathbb{R}^n )$, and $C^{k,\beta}_{q}(\overline{\mathbb{S}[\Omega_Q]^{-}},\mathbb{R}^n )$ as Banach subspaces of $C^{k}_{b}(\overline{\mathbb{S}[\Omega_Q]},\mathbb{R}^n )$,  of $C^{k,\beta}_{b}(\overline{\mathbb{S}[\Omega_Q]},\mathbb{R}^n )$,  of $C^{k}_{b}(\overline{\mathbb{S}[\Omega_Q]^{-}},\mathbb{R}^n )$, and of $C^{k,\beta}_{b}(\overline{\mathbb{S}[\Omega_Q]^{-}},\mathbb{R}^n )$, respectively. 

\section{Preliminaries of periodic potential theory for the Lam\'e equations}\label{prel}

In order to construct the solution of problem \eqref{bvp}, we will exploit a periodic version of potential theory for the Lam\'e equations and we begin by introducing some notation and tools.

We start by denoting by $S_{n}$  the function from  
${\mathbb{R}}^{n}\setminus\{0\}$ to
${\mathbb{R}}$ defined by 
\[
S_{n}(x)\equiv
\left\{
\begin{array}{lll}
\frac{1}{s_{n}}\log |x| \qquad &   \forall x\in 
{\mathbb{R}}^{n}\setminus\{0\},\quad & {\mathrm{if}}\ n=2\,,
\\
\frac{1}{(2-n)s_{n}}|x|^{2-n}\qquad &   \forall x\in 
{\mathbb{R}}^{n}\setminus\{0\},\quad & {\mathrm{if}}\ n>2\,,
\end{array}
\right.
\]
where $s_{n}$ denotes the $(n-1)$-dimensional measure of 
$\partial{\mathbb{B}}_{n}(0,1)$. $S_{n}$ is well-known to be a 
fundamental solution of the Laplace operator $\Delta=\sum_{j=1}^n\partial_{x_j}^2$.

We denote by $\Gamma_{n,\omega}(\cdot)$ the matrix valued function from $\mathbb{R}^n \setminus \{0\}$ to $M_{n}(\mathbb{R})$ that takes $x$ to the matrix $\Gamma_{n,\omega}(x)$ with $(i,j)$ entry defined by
\[
\Gamma_{n,\omega,i}^j(x)\equiv \frac{\omega+2}{2(\omega+1)}\delta_{i,j}S_n(x)-\frac{\omega}{2(\omega+1)}\frac{1}{s_n}\frac{x_i x_j}{|x|^n}\qquad\forall (i,j)\in\{1,\dots,n\}^2\,,
\]
where $\delta_{i,j}=1$ if $i=j$, and $\delta_{i,j}=0$ if $i \neq j$.  It is  well known that $\Gamma_{n,\omega}$ is a fundamental solution of the operator
\[
L[\omega]\equiv \Delta+\omega \nabla \mathrm{div}\,.
\]
We observe that the classical operator of linearized homogenous isotropic elastostatics equals $L[\omega]$ times the second constant of Lam\'e, and that $L[\omega]u=\mathrm{div} \,T(\omega,Du)$ for all regular vector valued functions $u$, and that the classical fundamental solution of the operator of linearized homogenous and isotropic elastostatics equals $\Gamma_{n,\omega}$ times the reciprocal of the second constant of Lam\'e (cf., {e.g.}, Kupradze {\it et al.~}\cite{KuGeBaBu79}). We find also convenient to set
\[
\Gamma_{n,\omega}^j\equiv \bigl(\Gamma_{n,\omega,i}^j\bigr)_{i \in \{1,\dots,n\}}\,,
\]
which we think as a column vector for all $j\in\{1,\dots,n\}$.  

To construct periodic elastic layer potentials, we need to use a periodic fundamental solution for the Lam\'e equations. In the following theorem (see \cite[Thm.~3.1]{DaMu14}) we introduce a periodic analog of the fundamental solution of $L[\omega]$ (cf., {e.g.}, Ammari and Kang \cite[Lemma 9.21]{AmKa07}, Ammari {\it et al.~}\cite[Lemma 3.2]{AmKaLi06}). To do so we need the following notation.  We denote by ${\mathcal{S}}({\mathbb{R}}^{n},\mathbb{C})$ the Schwartz space of complex valued rapidly decreasing functions. $\mathcal{S}'(\mathbb{R}^n,\mathbb{C})$ denotes the space of complex tempered distributions and $M_{n}\bigl(\mathcal{S}'(\mathbb{R}^n,\mathbb{C})\bigr)$ denotes the set of $n\times n$ matrices with entries in $\mathcal{S}'(\mathbb{R}^n,\mathbb{C})$. If $y\in{\mathbb{R}}^{n}$ and $f$ is a function defined in ${\mathbb{R}}^{n}$, we set
$\tau_{y}f(x)\equiv f(x-y)$ for all $x\in {\mathbb{R}}^{n}$. If $u\in \mathcal{S}'(\mathbb{R}^n,\mathbb{C})$, then we set
\[
<\tau_{y}u,f>\equiv<u,\tau_{-y}f>\qquad\forall f\in  \mathcal{S}(\mathbb{R}^n,\mathbb{C})\,.
\]
Finally, $L^1_{\mathrm{loc}}(\mathbb{R}^n)$ denotes the space of (equivalence classes of) locally summable measurable functions from $\mathbb{R}^n$ to $\mathbb{R}$.

\begin{theorem}
\label{psper}
Let $\Gamma_{n,\omega}^{q}\equiv (\Gamma_{n,\omega,j}^{q,k})_{(j,k)\in\{1,\dots,n\}^2}$ be the element of $M_{n}\bigl(\mathcal{S}'(\mathbb{R}^n,\mathbb{C})\bigr)$ with $(j,k)$ entry defined by
\[
\begin{split}
\Gamma_{n,\omega,j}^{q,k}\equiv \sum_{z \in \mathbb{Z}^n \setminus \{0\}} \frac{1}{4 \pi^2 |Q|  |q^{-1}z|^2}\Biggl[ -\delta_{j,k}+\frac{\omega}{\omega+1}&\frac{(q^{-1}z)_j(q^{-1}z)_k}{|q^{-1}z|^2}\Biggr]E_{2 \pi iq^{-1}z} \\ &\qquad \forall (j,k) \in \{1,\dots,n\}^2\,,
\end{split}
\]
where $E_{2\pi i q^{-1} z}$ is the function from $\mathbb{R}^n$ to $\mathbb{C}$ defined by
\[
E_{2\pi i q^{-1} z}(x)\equiv e^{2\pi i (q^{-1} z)\cdot x}
\qquad  \forall x\in{\mathbb{R}}^{n}
\] for all $z\in\mathbb{Z}^n$.
 Then the following statements hold.
\begin{enumerate}
\item[(i)]
\[
\tau_{q_{ll}e_l}\Gamma_{n,\omega,j}^{q,k}=\Gamma_{n,\omega,j}^{q,k} \qquad \forall l  \in \{1,\dots,n\}\,,
\]
for all $(j,k) \in \{1,\dots,n\}^2$.
\item[(ii)]
\[
L[\omega] \Gamma_{n,\omega}^{q}=\sum_{z \in \mathbb{Z}^n}\delta_{qz}I_n-\frac{1}{ |Q|}I_n \qquad \text{$\mathrm{in}$ $M_{n}\bigl(\mathcal{S}'(\mathbb{R}^n,\mathbb{C})\bigr)$}\,,
\]
where $\delta_{qz}$ denotes the Dirac measure with mass at $qz$ for all $z \in \mathbb{Z}^n$.
\item[(iii)] $\Gamma_{n,\omega}^{q}$ is real analytic from $\mathbb{R}^n \setminus q\mathbb{Z}^n$ to $M_n(\mathbb{R})$.
\item[(iv)] The difference $\Gamma_{n,\omega}^{q}-\Gamma_{n,\omega}$ can be extended to a real analytic function from $(\mathbb{R}^n \setminus q \mathbb{Z}^n) \cup\{0\}$ to $M_n(\mathbb{R})$ which we denote by $R^q_{n,\omega}$. Moreover
\[
L[\omega] R^q_{n,\omega}=\sum_{z \in \mathbb{Z}^n\setminus\{0\}}\delta_{qz}I_n-\frac{1}{ |Q|}I_n 
\]
in the sense of distributions.
\item[(v)] $\Gamma_{n,\omega,j}^{q,k}$ is real valued and is in $L^1_{\mathrm{loc}}(\mathbb{R}^n)$, for all $(j,k)\in \{1,\dots,n\}^2$.
\item[(vi)] $\Gamma_{n,\omega}^q(x)=\Gamma_{n,\omega}^{q}(-x)$ for all $x \in \mathbb{R}^n \setminus q\mathbb{Z}^n$.
\end{enumerate}
\end{theorem} 

\begin{remark}\label{rem:fundsol}
We observe that constructions similar to that of Theorem \ref{psper} have been used in \cite{Mu12} for a periodic fundamental solution of the Laplace equation, in \cite{LaMu18} for the Helmholtz equation, and in Luzzini \cite{Lu20} for the heat equation.
\end{remark}

Now that we have defined a periodic analog of the fundamental solution for the Lam\'e equations, we find  convenient to set
\[
\Gamma_{n,\omega}^{q,j}\equiv \bigl(\Gamma_{n,\omega,i}^{q,j}\bigr)_{i \in \{1,\dots,n\}}\,,\qquad
R_{n,\omega}^{q,j}\equiv \bigl(R_{n,\omega,i}^{q,j}\bigr)_{i \in \{1,\dots,n\}}\,,
\]
which we think as column vectors for all $j\in\{1,\dots,n\}$.  

We are now in the position to introduce the periodic single layer potential. To define it, it is sufficient to replace in the definition of the standard single layer potential for the Lam\'e equation the fundamental solution $\Gamma_{n,\omega}$ with its periodic analog $\Gamma^q_{n,\omega}$. So, if $\mu \in C^{0,\alpha}(\partial {\Omega_Q},\mathbb{R}^n)$, then we denote by $v_q[\omega, \mu]$ the periodic single layer potential, namely  the function from $\mathbb{R}^n$ to $\mathbb{R}^n$ defined by 
\[
v_q[\omega, \mu](x)\equiv \int_{\partial {\Omega_Q}}\Gamma^q_{n,\omega}(x-y)\mu(y)\,d\sigma_y \qquad \forall x \in \mathbb{R}^n\,.
\]
We note here that the fundamental solution $\Gamma^q_{n,\omega}$ takes values in $M_n(\mathbb{R})$ (cf.~Theorem \ref{psper} (ii) and (iii)). We also find convenient to set
\begin{equation}\label{eq:vqast}
W_{q}^\ast[\omega, \mu](x)\equiv \int_{\partial {\Omega_Q}}\sum_{l=1}^n \mu_{l}(y)T(\omega,D\Gamma_{n,\omega}^{q,l}(x-y))\nu_{{\Omega_Q}}(x)\,d\sigma_y \quad \forall x \in \partial {\Omega_Q}\,.
\end{equation}

In order to use the periodic single layer potential $v_q[\omega, \mu]$ to solve problem \eqref{bvp}, in the following theorem we present some properties of $v_q[\omega, \mu]$ (see \cite[Thm.~3.2]{DaMu14}).
\begin{theorem}
\label{sperpot}
The following statements hold.
\begin{enumerate}
\item[(i)] If $\mu\in C^{0,\alpha}(\partial{\Omega_Q},\mathbb{R}^n)$, then $v_{q}[\omega,\mu]$ is $q$-periodic and 
\[
L[\omega]v_{q}[\omega,\mu](x)
=
-\frac{1}{|Q|}\int_{\partial{\Omega_Q}}\mu \,d\sigma
\]
for all  $x\in 
 {\mathbb{R}}^{n}\setminus\partial{\mathbb{S}}[{\Omega_Q}]$.
\item[(ii)]  If $\mu\in C^{m-1,\alpha}(\partial{\Omega_Q},\mathbb{R}^n)$, then the function 
$v^{+}_{q}[\omega,\mu]\equiv v_{q}[\omega,\mu]_{|\overline{\mathbb{S}[\Omega_Q]}}$ belongs to $C^{m,\alpha}_{q}(\overline{\mathbb{S}[\Omega_Q]},\mathbb{R}^n)$ and the operator
 that takes $\mu$ to 
$v^{+}_{q}[\omega,\mu]  $ is  continuous from $C^{m-1,\alpha}(\partial{\Omega_Q},\mathbb{R}^n)$ to $C^{m,\alpha}_{q}(\overline{\mathbb{S}[{\Omega_Q}]},\mathbb{R}^n)$. 
\item[(iii)]  If $\mu\in C^{m-1,\alpha}(\partial{\Omega_Q},\mathbb{R}^n)$, then the function 
$v^{-}_{q}[\omega,\mu]\equiv v_{q}[\omega,\mu]_{|\overline{\mathbb{S}[\Omega_Q]^{-}}}$ belongs to $C^{m,\alpha}_{q}
(\overline{\mathbb{S}[\Omega_Q]^{-}},\mathbb{R}^n)$ and the operator that takes $\mu$ to $v^{-}_{q}[\omega,\mu]$   is  continuous from  $ C^{m-1,\alpha}(\partial{\Omega_Q},\mathbb{R}^n)$ to $C^{m,\alpha}_{q}
(\overline{\mathbb{S}[\Omega_Q]^{-}},\mathbb{R}^n)$.
\item[(iv)]  The operator that takes $\mu$ to $W_{q}^\ast[\omega,\mu]$ is  continuous from   $C^{m-1,\alpha}(\partial{\Omega_Q},\mathbb{R}^n)$ to itself, and we have 
\[
T\bigl(\omega,Dv_{q}^{\pm}[\omega,\mu](x)\bigr)\nu_{{\Omega_Q}}(x)=\mp\frac{1}{2}\mu(x)+W_{q}^\ast[\omega,\mu](x) \qquad \forall x \in \partial{\Omega_Q}\, ,
\]
for all $\mu \in C^{m-1,\alpha}(\partial{\Omega_Q},\mathbb{R}^n)$.
\item[(v)] We have 
\[
\int_{\partial {\Omega_Q}}W_{q}^\ast[\omega,\mu]\,d\sigma= \left(\frac{1}{2}-\frac{|{\Omega_Q}|}{|Q|}\right)\int_{\partial {\Omega_Q}}\mu\,d\sigma\,.
\]
\end{enumerate}
\end{theorem}

To solve problem \eqref{bvp} we will need to exploit some properties of the auxiliary operator $\frac{1}{2}I +W_{q}^\ast[\omega,\cdot]$. Therefore, we recall the following result of  \cite[Prop.~3.4]{DaMu14}.

\begin{proposition}\label{prop:bijbdry}
The operator $\frac{1}{2}I +W_{q}^\ast[\omega,\cdot]$ is a linear homeomorphism from the space $C^{m-1,\alpha}(\partial {\Omega_Q},\mathbb{R}^n)$ to itself.
\end{proposition}

In the following proposition, we recall the representation formula of  \cite[Prop.~3.5]{DaMu14} for a periodic function $u$ defined on the set $\overline{\mathbb{S}[{\Omega_Q}]^{-}}$ and such that $L[\omega]u=0$. To do so we need to introduce the following set of functions with zero integral on $\partial \Omega_Q$:
\[
C^{m-1,\alpha}(\partial {\Omega_Q},\mathbb{R}^n)_0\equiv \left\{f \in C^{m-1,\alpha}(\partial {\Omega_Q},\mathbb{R}^n)\colon \int_{\partial {\Omega}}f \, d\sigma=0\right\}\,.
\]

Then we are in the position to formulate the following result, which states that any function $u \in C^{m,\alpha}_q(\overline{\mathbb{S}[{\Omega_Q}]^{-}},\mathbb{R}^n)$ such that $L[\omega]u(x)=0$ for all $x \in \mathbb{S}[{\Omega_Q}]^-$ can be represented as the sum of a periodic single layer potential and a constant.

\begin{proposition}\label{prop:rep}
Let $u \in C^{m,\alpha}_q(\overline{\mathbb{S}[{\Omega_Q}]^{-}},\mathbb{R}^n)$. Assume that
\[
L[\omega]u(x)=0 \qquad \forall x \in \mathbb{S}[{\Omega_Q}]^-\,.
\]
Then there exists a unique pair $(\mu,c) \in C^{m-1,\alpha}(\partial {\Omega_Q},\mathbb{R}^n)_0\times \mathbb{R}^n$ such that
\[
u(x)=v^-_q[\omega,\mu](x)+c \qquad \forall x \in \overline{\mathbb{S}[{\Omega_Q}]^{-}}\,.
\]
\end{proposition}

\section{Existence and uniqueness for the solution of problem (\ref{bvp})}
\label{exun}

In this section, we prove existence and uniqueness for the solution of problem \eqref{bvp}. First, we start with this uniqueness result for the solution of an homogeneous Robin problem with periodicity condition (\textit{i.e.}, with $B=0$).

\begin{proposition}\label{prop:uneu}
 Let $u \in  C^{m,\alpha}_q(\overline{\mathbb{S}[\Omega_Q]^{-}},\mathbb{R}^n)$ be such that
 \begin{equation}\label{bvp:uneu}
 \left \lbrace 
 \begin{array}{ll}
L[\omega]u = 0 & \mathrm{in}\  {\mathbb{S}} [\Omega_Q]^-\,, \\
u(x+qe_j) =u(x) &  \textrm{$\forall x \in \overline{\mathbb{S}[\Omega_Q]^{-}}, \forall j\in \{1,\dots,n\}$}, \\
a(x)T(\omega,Du(x))\nu_{\Omega_Q}(x)+ b(x)u(x)=0 & \textrm{$\forall x \in \partial \Omega_Q$}\,.
 \end{array}
 \right.
 \end{equation}
Then $u(x)=0$ for all $x \in \overline{\mathbb{S}[\Omega_Q]^{-}}$.
\end{proposition}
\begin{proof} 
 By the periodicity of $u$ we have $\int_{\partial Q}u^tT(\omega,Du)\nu_Q\,d\sigma=0$. Moreover, the third condition in \eqref{bvp:uneu} implies that
\[
T(\omega,Du(x))\nu_{\Omega_Q}(x)=- a^{-1}(x)b(x)u(x)  \qquad \textrm{$\forall x \in \partial \Omega_Q$.}
\]
Thus  the Divergence Theorem implies that
\[
0 \leq \int_{ Q\setminus \overline{\Omega_Q}}\mathrm{tr} \bigl(T(\omega,Du)D^tu\bigr)\,dx=-\int_{\partial {\Omega_Q}}u^tT(\omega,Du)\nu_{\Omega_Q} \,d\sigma=\int_{\partial {\Omega_Q}}u^ta^{-1}bu  \,d\sigma \leq 0\,,
\]
since, by assumption on $a, b$, we have
\[
u^t(x)a^{-1}(x)b(x)u(x)\leq 0 \qquad \forall x \in \partial \Omega_Q\, .
\]
Then $\mathrm{tr} \bigl(T(\omega,Du)D^tu\bigr)=0$ in $Q\setminus \overline{\Omega_Q}$, and by arguing as in \cite[Proposition 2.1]{DaLa10a}, we can prove that there exist a skew symmetric matrix $A \in M_{n}(\mathbb{R})$ and $c \in \mathbb{R}^n$, such that
\[
u(x)=Ax+c \qquad \forall x \in \overline{Q}\setminus \overline{\Omega_Q}\,.
\]
By the periodicity of $u$, we have
\[
A qe_k=u(qe_k)-u(0)=0 \qquad \forall k \in \{1,\dots,n\}\,.
\]
Accordingly, $A=0$. Hence, $u(x)=c$ for all $x \in \overline{Q}\setminus \overline{\Omega_Q}$, and thus, by periodicity, $u(x)=c$ for all $x \in \overline{\mathbb{S}[\Omega_Q]^{-}}$. Then again by the third condition in \eqref{bvp:uneu} we deduce that
\[
0+b(x)c=0 \qquad \forall x \in \partial \Omega_Q\, ,
\]
and thus, in particular
\[
b(x_0)c=0\, ,
\]
which together with $\mathrm{det}\, b(x_0)\neq 0$ implies $c=0$.
\end{proof}

In Proposition~\ref{prop:bij} below, we show that the operator
\[
(\mu,c)\mapsto \frac{1}{2}\mu +W_{q}^\ast[\omega,\mu]_{|\partial \Omega_Q}+a^{-1}b\Big(v_{q}[\omega, \mu]+c\Big)
\]
is a linear homeomorphism from $C^{m-1,\alpha}(\partial\Omega_Q,\mathbb{R}^n)_0 \times \mathbb{R}^n$ to $C^{m-1,\alpha}(\partial\Omega_Q,\mathbb{R}^n)$ (cf.~definition \eqref{eq:vqast}). As we shall see, such operator appears if we want to solve a periodic Robin boundary value problem for the Lam\'e equations in terms of a single layer potential plus a costant. However, before proving Proposition~\ref{prop:bij}, we need the following intermediate step, which is the periodic counterpart of \cite[Prop.~4.4]{DaMu14}.

\begin{lemma}\label{lem:bijD}
Let $D\equiv (d_{ij}(\cdot))_{(i,j) \in \{1,\dots, n\}^2} \in C^{m-1,\alpha}(\partial \Omega_Q,M_{n}(\mathbb{R}))$ be such that the matrix
\[
\int_{\partial \Omega_Q}D(y)\,d\sigma_y \equiv \Biggl(\int_{\partial \Omega_Q}d_{ij}(y)\,d\sigma_y\Biggr)_{(i,j)\in \{1,\dots,n\}^2}
\]
is invertible. Then the operator from $C^{m-1,\alpha}(\partial \Omega_Q,\mathbb{R}^n)_0\times \mathbb{R}^n$ to $C^{m-1,\alpha}(\partial \Omega_Q,\mathbb{R}^n)$ that takes $(\mu,c)$ to the function
\[
\frac{1}{2}\mu+W_{q}^\ast[\omega,\mu]+D c
\]
is a linear homeomorphism.
\end{lemma}
\begin{proof} We start by denoting by $\mathcal{H}$ the linear operator from $C^{m-1,\alpha}(\partial \Omega_Q,\mathbb{R}^n)_0\times \mathbb{R}^n$ to $C^{m-1,\alpha}(\partial \Omega_Q,\mathbb{R}^n)$ defined by
\[
\mathcal{H}[\mu,c]\equiv \frac{1}{2}\mu+W_{q}^\ast[\omega,\mu]+D c\, ,
\] 
for all $(\mu,c) \in C^{m-1,\alpha}(\partial \Omega_Q,\mathbb{R}^n)_0\times \mathbb{R}^n$. By Theorem \ref{sperpot}, $\mathcal{H}$ is a linear and continuous operator from  $C^{m-1,\alpha}(\partial \Omega_Q,\mathbb{R}^n)_0\times \mathbb{R}^n$ to $C^{m-1,\alpha}(\partial \Omega_Q,\mathbb{R}^n)$. To prove the lemma, we need to show that $\mathcal{H}$ is a linear homeomorphism. Thus, by the Open Mapping Theorem, it suffices to prove that it is a bijection. So let $\psi \in C^{m-1,\alpha}(\partial \Omega_Q,\mathbb{R}^n)$. We need to prove that there exists a unique pair $(\mu,c) \in  C^{m-1,\alpha}(\partial \Omega_Q,\mathbb{R}^n)_0\times \mathbb{R}^n$ such that
\begin{equation}\label{eq:bijD1}
\frac{1}{2}\mu(x)+W_{q}^\ast[\omega,\mu](x)+D(x) c=\psi(x) \qquad \forall x \in \partial \Omega_Q\,.
\end{equation}
We first prove uniqueness. Let us assume that the pair $(\mu,c) \in  C^{m-1,\alpha}(\partial \Omega_Q,\mathbb{R}^n)_0\times \mathbb{R}^n$ solves equation \eqref{eq:bijD1}. By integrating both sides of equation \eqref{eq:bijD1}, and by the identity
\begin{equation}\label{eq:bijD0}
\int_{\partial \Omega_Q}\Bigl(\frac{1}{2}\mu(x)+W_{q}^\ast[\omega,\mu](x)\Bigr)\,d\sigma_x= \left(1-\frac{|{\Omega_Q}|}{|Q|}\right)\int_{\partial \Omega_Q}\mu(x)\,d\sigma_x
\end{equation}
(cf. Theorem \ref{sperpot} (v)), we obtain
\[
\Bigl(\int_{\partial \Omega_Q}D(x)\,d\sigma_x\Bigr)c=\int_{\partial \Omega_Q}\psi(x)\,d\sigma_x\,,
\]
and thus
\begin{equation}\label{eq:bijD2}
c=\Bigl(\int_{\partial \Omega_Q}D(x)\,d\sigma_x\Bigr)^{-1}\int_{\partial \Omega_Q}\psi(x)\,d\sigma_x\,.
\end{equation}
As a consequence, by Proposition \ref{prop:bijbdry}, $\mu$ is the unique solution in $C^{m-1,\alpha}(\partial \Omega_Q,\mathbb{R}^n)$ of equation
\begin{equation}\label{eq:bijD3}
\frac{1}{2}\mu(x)+W_{q}^\ast[\omega,\mu](x)=\psi(x)-D(x) \Bigl(\int_{\partial \Omega_Q}D(y)\,d\sigma_y\Bigr)^{-1}\int_{\partial \Omega_Q}\psi(y)\,d\sigma_y \qquad \forall x \in \partial \Omega_Q\, .
\end{equation}
We also note that by equality \eqref{eq:bijD0} the unique solution of equation \eqref{eq:bijD3} is in the space $C^{m-1,\alpha}(\partial \Omega_Q,\mathbb{R}^n)_0$. Hence uniqueness follows. In order to prove existence, it suffices to observe that the pair $(\mu,c) \in  C^{m-1,\alpha}(\partial \Omega_Q,\mathbb{R}^n)_0\times \mathbb{R}^n$ identified by equations \eqref{eq:bijD2}, \eqref{eq:bijD3} solves equation \eqref{eq:bijD1} (cf. Proposition \ref{prop:bijbdry}).
\end{proof}

We are now ready to prove the following.

\begin{proposition}\label{prop:bij}
The operator from $C^{m-1,\alpha}(\partial {\Omega_Q},\mathbb{R}^n)_0 \times \mathbb{R}^n$ to  $C^{m-1,\alpha}(\partial {\Omega_Q},\mathbb{R}^n)$ that takes a pair $(\mu,c)$ to 
\begin{equation}\label{eq:bij1}
\frac{1}{2}\mu +W_{q}^\ast[\omega,\mu]+a^{-1}b\Big(v_{q}[\omega, \mu]_{|\partial \Omega_Q}+c\Big)
\end{equation}
 is a linear homeomorphism.
\end{proposition}
\begin{proof} We first note that by Lemma \ref{lem:bijD}   the operator from $C^{m-1,\alpha}(\partial {\Omega_Q},\mathbb{R}^n)_0 \times \mathbb{R}^n$ to  $C^{m-1,\alpha}(\partial {\Omega_Q},\mathbb{R}^n)$ that takes a pair $(\mu,c)$ to 
\[
\frac{1}{2}\mu +W_{q}^\ast[\omega,\mu]+a^{-1}bc
\]
 is a linear homeomorphism. Moreover,  $v_q[\omega, \cdot]_{|\partial \Omega_Q}$ maps continuously $C^{m-1,\alpha}(\partial {\Omega_Q},\mathbb{R}^n)_0$ into  $C^{m,\alpha}(\partial {\Omega_Q},\mathbb{R}^n)$, which is compactly embedded into $C^{m-1,\alpha}(\partial {\Omega_Q},\mathbb{R}^n)$. Therefore, the operator in \eqref{eq:bij1} is a compact perturbation of a Fredholm operator of index $0$, and thus is itself a Fredholm operator of index $0$.  By the Open Mapping Theorem, in order to prove the operator in \eqref{eq:bij1}  is a linear homeomorphism, it suffices to show that it is a bijection. Also, by the Fredholm theory, to show that it is surjective, it is enough to prove the injectivity. To do so, we verify that if 
\begin{equation}\label{eq:bij2}
\frac{1}{2}\mu +W_{q}^\ast[\omega,\mu]+a^{-1}b\Big(v_{q}[\omega, \mu]_{|\partial \Omega_Q}+c\Big)=0\, ,
\end{equation}
 then $(\mu, c)=(0,0)$. So  let $(\mu,c)$ be such that \eqref{eq:bij2} holds. Then the function $v_{q}^-[\omega, \mu]+c$ is a solution of boundary value problem \eqref{bvp:uneu}, and thus Proposition \ref{prop:uneu} implies that  $v_{q}^-[\omega, \mu]+c=0$ in $\overline{\mathbb{S}[\Omega_Q]^-}$. Finally, Proposition \ref{prop:rep} implies that $(\mu,c)=(0,0)$.
\end{proof}

We are now ready to prove our main result on the solvability of the Robin boundary value problem.

\begin{theorem}\label{thm:exbvp}
There exists a unique function $u  \in  C^{m,\alpha}_{\mathrm{loc}}(\overline{\mathbb{S}[\Omega_Q]^{-}},\mathbb{R}^n)$ such that
 \begin{equation}\label{eq:exbvp1}
 \left \lbrace 
 \begin{array}{ll}
L[\omega]u = 0 & \mathrm{in}\  {\mathbb{S}} [\Omega_Q]^-\,, \\
u(x+qe_j) =u(x)+Be_j &  \textrm{$\forall x \in \overline{\mathbb{S}[\Omega_Q]^{-}}, \forall j\in \{1,\dots,n\}$}, \\
a(x)T(\omega,Du(x))\nu_{\Omega_Q}(x)+ b(x)u(x)=g(x) & \textrm{$\forall x \in \partial \Omega_Q$}\,.
 \end{array}
 \right.
 \end{equation}
 Moreover,
\begin{equation}\label{eq:exbvp2}
 u(x)=v_{q}^-[\omega,\mu](x)+c+Bq^{-1}x \qquad \forall x \in \overline{\mathbb{S}[\Omega_Q]^{-}}\, ,
\end{equation}
where $(\mu,c)$ is the unique pair in $C^{m-1,\alpha}(\partial\Omega_Q,\mathbb{R}^n)_0\times \mathbb{R}^n$ such that
\begin{equation}\label{eq:exbvp3}
\begin{split}
\frac{1}{2}\mu(x) &+W_{q}^\ast[\omega,\mu](x)+a^{-1}(x)b(x)\Big(v_{q}[\omega, \mu]_{|\partial \Omega_Q}(x)+c\Big)\\&=a^{-1}(x)g(x)-T(\omega,Bq^{-1})\nu_{\Omega_Q}(x)- a^{-1}(x)b(x)Bq^{-1}x\qquad \forall x \in \partial \Omega_Q\, .
\end{split}
\end{equation}
\end{theorem}
\begin{proof} We first consider uniqueness. So let $u'$, $u''$ be two solutions in $C^{m,\alpha}_{\mathrm{loc}}(\overline{\mathbb{S}[\Omega_Q]^{-}},\mathbb{R}^n)$ of problem \eqref{eq:exbvp1}. Then we set 
\[
\tilde{u}\equiv u'-u''\, ,
\]
and we note that
\[
 \left \lbrace 
 \begin{array}{ll}
L[\omega]\tilde{u} = 0 & \mathrm{in}\  {\mathbb{S}} [\Omega_Q]^-\,, \\
\tilde{u}(x+qe_j) =\tilde{u}(x) &  \textrm{$\forall x \in \overline{\mathbb{S}[\Omega_Q]^{-}}, \forall j\in \{1,\dots,n\}$}, \\
a(x)T(\omega,D\tilde{u}(x))\nu_{\Omega_Q}(x)+ b(x)\tilde{u}(x)=0 & \textrm{$\forall x \in \partial \Omega_Q$}\,.
 \end{array}
 \right.
\]
Accordingly, Proposition \ref{prop:uneu} implies that $\tilde{u}=0$ and thus $u'=u''$.

We now turn to prove existence. We first note that if $u_\# \in C^{m,\alpha}_{q}(\overline{\mathbb{S}[\Omega_Q]^{-}},\mathbb{R}^n)$ is such that
 \begin{equation}\label{eq:exbvp4}
 \left \lbrace 
 \begin{array}{ll}
L[\omega]u_\# = 0 & \mathrm{in}\  {\mathbb{S}} [\Omega_Q]^-\,, \\
u_\#(x+qe_j) =u_\#(x) &  \textrm{$\forall x \in \overline{\mathbb{S}[\Omega_Q]^{-}}, \forall j\in \{1,\dots,n\}$}, \\
a(x)T(\omega,Du_\#(x))\nu_{\Omega_Q}(x)+ b(x)u_\#(x)&\\
\qquad=g(x)-a(x)T(\omega,Bq^{-1})\nu_{\Omega_Q}(x)- b(x)Bq^{-1}x & \textrm{$\forall x \in \partial \Omega_Q$}\,,
 \end{array}
 \right.
 \end{equation}
 then the function
 \[
 x\mapsto u_\#(x)+Bq^{-1}x
 \]
 is a solution of problem \eqref{eq:exbvp1}. Therefore, in order to prove that the function in \eqref{eq:exbvp2} solves problem \eqref{eq:exbvp1}, it suffices to show that 
 \begin{equation}\label{eq:exbvp5}
 v_{q}^-[\omega,\mu]+c
 \end{equation}
solves problem \eqref{eq:exbvp4}, where $(\mu,c) \in C^{m-1,\alpha}(\partial\Omega_Q,\mathbb{R}^n)_0\times \mathbb{R}^n$ is such that \eqref{eq:exbvp3} holds. We first note that Proposition \ref{prop:bij} ensures the existence and uniqueness of a pair $(\mu,c) \in C^{m-1,\alpha}(\partial\Omega_Q,\mathbb{R}^n)_0\times \mathbb{R}^n$ which solves \eqref{eq:exbvp3}. Accordingly, by Theorem \ref{sperpot}, we immediately verify that the function in \eqref{eq:exbvp5} satisfies the first two conditions of problem \eqref{eq:exbvp1}. Moreover, by a straightforward computation we also verify that equation \eqref{eq:exbvp3} implies the validity of the third condition of problem \eqref{eq:exbvp1} (see Theorem \ref{sperpot}). Thus the proof is complete.

\end{proof}

\section{A remark on a nonlinear Robin-type traction problem}

In this section, we show how the integral equation method of the previous sections can be exploited in order to solve a nonlinear Robin-type traction boundary value problem in a periodic domain. To do so,  we consider a function $G\in C^{0}(\partial \Omega_Q\times \mathbb{R}^n,\mathbb{R}^n)$ such that the non-autonomous composition operator $F_G$ defined by
\[
F_G[v](x)=G(x,v(x)) \qquad \forall x \in \partial \Omega_Q\, , \forall v \in C^{0}(\partial \Omega_Q,\mathbb{R}^n)\, ,
\]
maps $C^{m-1,\alpha}(\partial \Omega_Q,\mathbb{R}^n)$ to itself. Then we introduce the problem
\begin{equation}\label{eq:nlbvp1}
 \left \lbrace 
 \begin{array}{ll}
L[\omega]u = 0 & \mathrm{in}\  {\mathbb{S}} [\Omega_Q]^-\,, \\
u(x+qe_j) =u(x)+Be_j &  \textrm{$\forall x \in \overline{\mathbb{S}[\Omega_Q]^{-}}, \forall j\in \{1,\dots,n\}$}, \\
T(\omega,Du(x))\nu_{\Omega_Q}(x)= G(x,u(x)) & \textrm{$\forall x \in \partial \Omega_Q$}\,.
 \end{array}
 \right.
 \end{equation}

In the theorem below, we show an integral equation formulation of the nonlinear problem \eqref{eq:nlbvp1}.

\begin{theorem}\label{thm:nlbvp}
The map from $C^{m-1,\alpha}(\partial\Omega_Q,\mathbb{R}^n)_0\times \mathbb{R}^n$  to $C^{m,\alpha}_{\mathrm{loc}}(\overline{\mathbb{S}[\Omega_Q]^{-}},\mathbb{R}^n)$ that takes a pair $(\mu,c)$ to the function
\begin{equation}\label{eq:nlbvp2}
 v_{q}^-[\omega,\mu](x)+c+Bq^{-1}x \qquad \forall x \in \overline{\mathbb{S}[\Omega_Q]^{-}}\, ,
\end{equation}
is a bijection from the set of solutions in $C^{m-1,\alpha}(\partial\Omega_Q,\mathbb{R}^n)_0\times \mathbb{R}^n$ such that
\begin{equation}\label{eq:nlbvp3}
\begin{split}
&\frac{1}{2}\mu(x) +W_{q}^\ast[\omega,\mu](x)\\&=G\Big(x,v_{q}[\omega, \mu]_{|\partial \Omega_Q}(x)+c+Bq^{-1}x \Big)-T(\omega,Bq^{-1})\nu_{\Omega_Q}(x)\qquad \forall x \in \partial \Omega_Q\, ,
\end{split}
\end{equation} 
to the set of functions  $u  \in  C^{m,\alpha}_{\mathrm{loc}}(\overline{\mathbb{S}[\Omega_Q]^{-}},\mathbb{R}^n)$ which solve problem \eqref{eq:nlbvp1}.
\end{theorem}
\begin{proof} 
Assume that the function  $u  \in  C^{m,\alpha}_{\mathrm{loc}}(\overline{\mathbb{S}[\Omega_Q]^{-}},\mathbb{R}^n)$ solve problem \eqref{eq:nlbvp1}. Then by Lemma \ref{prop:rep}, there exists a unique pair $(\mu,c)$ in $C^{m-1,\alpha}(\partial \Omega_Q,\mathbb{R}^n)_0\times \mathbb{R}^n$ such that $u$ equals the functions defined by \eqref{eq:nlbvp2}. Then a simple computation based on Theorem \ref{sperpot}, shows that the pair $(\mu,c)$ must solve equation \eqref{eq:nlbvp3}. Conversely, one can easily show that if the pair $(\mu,c)$ of $C^{m-1,\alpha}(\partial \Omega_Q,\mathbb{R}^n)_0\times \mathbb{R}^n$ solves equation \eqref{eq:nlbvp3}, then the function defined in \eqref{eq:nlbvp2} is a solution of problem \eqref{eq:nlbvp1}.
\end{proof}

By Theorem \ref{thm:nlbvp} we can convert the nonlinear boundary value problem \eqref{eq:nlbvp1} into the nonlinear integral equation \eqref{eq:nlbvp3}. Then, by arguing as in \cite[Thm.~8]{DaLa10b}, we can show that under suitable growth conditions equation \eqref{eq:nlbvp3} admits (at least) a solution $(\mu,c)$ in $C^{m-1,\alpha}(\partial\Omega_Q,\mathbb{R}^n)_0\times \mathbb{R}^n$. As a consequence, we can deduce the existence of (at least) a solution of problem \eqref{eq:nlbvp1}.

\section*{Acknowledgement}
P.M.~and G.M.~acknowledge the support from EU through the H2020-MSCA-RISE-2020 project EffectFact,
Grant agreement ID: 101008140.   M.D.R. and P.M.~are   members of the Gruppo Nazionale per l'Analisi Matematica, la Probabilit\`a e le loro Applicazioni (GNAMPA) of the Istituto Nazionale di Alta Matematica (INdAM). P.M.~also acknowledges the support of  the grant ``Challenges in Asymptotic and Shape Analysis - CASA''  of the Ca' Foscari University of Venice. Part of the work was done while P.M.~was visiting Martin Dutko~at Rockfield Software Limited. P.M.~wishes to thank Martin Dutko for valuable discussions and Rockfield Software Limited for the kind hospitality. G.M.~acknowledges also Ser Cymru Future Generation Industrial Fellowship number AU224 -- 80761.

\end{document}